\documentclass[10pt,a4paper]{amsart}
\usepackage[utf8]{inputenc}
\usepackage{amsmath}
\usepackage{amsfonts}
\usepackage{amssymb}
\usepackage[english]{babel}

\newtheorem{prop}{Proposition}
\newtheorem{hip}{Conjecture}
\newtheorem{tw}{Theorem}
\newtheorem{pyt}{Question}

\theoremstyle{remark}
\newtheorem*{Convention}{Convention}

\title{A note on $p$-adic valuations of the Schenker sums}
\author{Piotr Miska}


\begin{document}

\setlength{\parindent}{10mm}
\maketitle

\begin{abstract}
A prime number $p$ is called a Schenker prime if there exists such $n\in\mathbb{N}_+$ that $p\nmid n$ and $p\mid a_n$, where $a_n = \sum_{j=0}^{n}\frac{n!}{j!}n^j$ is so-called Schenker sum. T. Amdeberhan, D. Callan and V. Moll formulated two conjectures concerning $p$-adic valuations of $a_n$ in case when $p$ is a Schenker prime. In particular, they asked whether for each $k\in\mathbb{N}_+$ there exists the unique positive integer $n_k<p^k$ such that $v_p(a_{m\cdot 5^k + n_k})\geq k$ for each nonnegative integer $m$. We prove that for every $k\in\mathbb{N}_+$ the inequality $v_5(a_n)\geq k$ has exactly one solution modulo $5^k$. This confirms the first conjecture stated by the mentioned authors. Moreover, we show that if $37\nmid n$ then $v_{37}(a_n)\leq 1$, what means that the second conjecture stated by the mentioned authors is not true.
\end{abstract}

\section{Introduction}

Questions concerning the behaviour of $p$-adic valuations of elements of integer sequences are interesting subjects of research in number theory. The knowledge of all $p$-adic valuations of a given number is equivalent to its factorization. Papers \cite{1}, \cite{2}, \cite{3}, \cite{5} and \cite{6} present interesting results concerning behaviour of $p$-adic valuations in some integer sequences.

Fix a prime number $p$. Every nonzero rational number $x$ can be written in the form $x=\frac{a}{b}p^t$, where $a\in\mathbb{Z}$, $b\in\mathbb{N}_+$, $\gcd(a,b)=1$ and $p\nmid ab$. Such a representation of $x$ is unique, thus the number $t$ is well defined. We call $t$ the $p$-adic valuation of the number $x$ and denote it by $v_p(x)$. By convention, $v_p(0)=+\infty$. In particular, if $x\in\mathbb{Q}\setminus\lbrace 0\rbrace$ then $\vert x\vert=\prod_{p \mbox{\scriptsize{ prime}}}p^{v_p(x)}$, where $v_p(x)\neq 0$ for finitely many prime numbers $p$. In the sequel by $s_d(n)$ we denote the sum of digits of positive integer $n$ in base $d$, i.e. if $n=\sum_{i=0}^m c_id^i$ is an expansion of $n$ in base $d$, then $s_d(n)=\sum_{i=0}^m c_i$.

In a recent paper T. Amdeberhan, D. Callan and V. Moll introduced the sequence of Schenker sums, which are defined in the following way:
\begin{equation*}
a_n = \sum_{j=0}^{n}\frac{n!}{j!}n^j, n\in\mathbb{N_+}.
\end{equation*}
The authors of \cite{1} obtained exact expression for 2-adic valuation of Schenker sums:
\begin{equation*}
v_2(a_n) = \begin{cases}
1, & \mbox{ when } 2\nmid n
\\n-s_2(n), & \mbox{ when }  2\mid n
\end{cases}.
\end{equation*}
Moreover, they proved two results concerning $p$-adic valuation of elements of sequence $ (a_n)_{n\in\mathbb{N}_+} $ when $p$ is an odd prime number:

\begin{prop}[Proposition 3.1 in \cite{1}]\label{prop1}
Let $p$ be an odd prime number and $n=pm$ for some $m\in\mathbb{N}$. Then:
\begin{equation*}
v_p(a_n)=\frac{n-s_p(n)}{p-1}=v_p(n!).
\end{equation*}
\end{prop}

\begin{prop}[Proposition 3.2 in \cite{1}]\label{prop2}
Let $p$ be an odd prime number and ${n = pm+r}$, where $0<r<p$. Then $p\mid{a_n}$ if and only if $p\mid{a_r}$.
\end{prop}

These propositions allow us to compute $p$-adic valuation of $a_n$, when $p\mid n$ or $p\nmid a_{n\mbox{\scriptsize{ mod }} p}$. This gives complete description of $p$-adic valuation of numbers $a_n$ for some prime numbers ($3, 7, 11, 17$, for example). However, the question concerning $p$-adic valuation in case, when $p\nmid{n}$ and $p\mid{a_n}$ for some postive integer $n$ is much more difficult. The first prime $p$ such that $p\nmid n$ and $p\mid a_n$ for some $n\in\mathbb{N_+}$ is $p=5$. We have $5\mid{a_{5m+2}}$ for every $m\in\mathbb{N}$. Let us observe that if $n\not\equiv 0,2 \pmod {5}$, then $5\nmid{a_n}$. According to numerical experiments, the authors of the paper [1] formulated a conjecture,  equivalent version of which is as follows:

\begin{hip}[Conjecture 4.6 in \cite{1}]\label{conj1}
Assume that $n_k$ is the unique natural number less than $5^k$ such that $5^k\mid{a_{m\cdot5^k+n_k}}, m\in\mathbb{N}$. Then there exists the unique number $n_{k+1}\in\{n_k, 5^k+n_k, 2\cdot5^k+n_k, 3\cdot5^k+n_k, 4\cdot5^k+n_k\}$ such that ${5^{k+1}\mid{a_{m\cdot5^{k+1}+n_{k+1}}}}$, $m\in\mathbb{N}$. In other words, for every $k\in\mathbb{N}$ the inequality $v_5(a_n)\geq{k}$ has the unique solution $n \pmod{5^k}$ with $5\nmid n$.
\end{hip}
Number $5$ is not the unique prime number $p$ such that $p\nmid n$ and $p\mid{a_n}$ for some $n\in\mathbb{N_+}$. The prime numbers which satisfy the condition above are called \emph{Schenker primes}. 

The first question which comes to mind is: what is the cardinality of the set of Schenker primes? We will prove the following proposition using modification of the Euclid's proof of infinitude of set of prime numbers:

\begin{prop}
There are infinitely many Schenker primes.
\end{prop}

\begin{proof}
Assume that there are only finitely many Schenker primes and let $p_1, p_2, \dots, p_s$ be the odd Schenker primes in ascending order. Since $a_1=2$, we thus obtain by proposition \ref{prop2} that $p_1, p_2, p_3, \dots, p_s \nmid a_{p_1p_2p_3...p_s+1}$. Let us put $t := p_1p_2p_3\dots p_s+1$ and note that it is an even number. By Proposition \ref{prop1} we have:
\begin{equation*}
2t! \leq t! \sum_{j=0}^{t} \frac{t^j}{j!} = a_{t} = \prod_{p \mbox{\scriptsize{ prime, }} p\mid t} p^{v_p(t!)} \leq t!
\end{equation*}
and it leads to contradiction.
\end{proof}

The main result of this paper is contained in the following theorem:
\begin{tw}\label{tw1}
Let $p$ be a prime number, let $n_k\in\mathbb{N}$ be such that $p\nmid{n_k}$, $p^k\mid{a_{n_k}}$ and
\begin{equation*}
q_{n_k,p} := a_{n_k+p}-a_{n_k}(n_k+p)^{n_k+2}n_k^{p-n_k-2}.
\end{equation*}
Then:
\begin{itemize}
\item if $q_{n_k,p}\not\equiv 0 \pmod{p^2}$, then there exists unique $n_{k+1}$ modulo $p^{k+1}$ for which $p^{k+1}\mid{a_{n_{k+1}}}$ and $n_{k+1}\equiv n_k \pmod{p^k}$;
\item if $q_{n_k,p}\equiv 0 \pmod{p^2}$ and $p^{k+1}\mid{a_{n_k}}$, then $p^{k+1}\mid{a_{n_{k+1}}}$ for any $n_{k+1}$ satisfying $n_{k+1}\equiv n_k \pmod{p^k}$;
\item if $q_{n_k,p}\equiv 0 \pmod{p^2}$ and $p^{k+1}\nmid{a_{n_k}}$, then $p^{k+1}\nmid{a_{n_{k+1}}}$ for any $n_{k+1}$ satisfying $n_{k+1}\equiv n_k \pmod{p^k}$.
\end{itemize}

Moreover, if $p\nmid{n_1}$, $p\mid{a_{n_1}}$ and $q_{n_1,p}\not\equiv 0 \pmod{p^2}$, then for any $k\in\mathbb{N}_+$ the inequality $v_p(a_n)\geq{k}$ has the unique solution $n_k$ modulo $p^k$ satisfying the congruence $n_k\equiv n_1 \pmod{p}$.
\end{tw}
The proof of this theorem is given in section 2.

The authors of \cite{1} stated another, more general conjecture concerning $p$-adic valuations of numbers $a_n$ when $p$ is an odd Schenker prime. The equivalent version of this conjecture is as follows:
\begin{hip}[Conjecture 4.12. in \cite{1}]\label{conj2}
Let $p$ be an odd Schenker prime. Then for every $k$ there exists the unique solution modulo $p^k$ of inequality $v_{p}(a_n)\geq k$ which is not congruent to $0$ modulo $p$.
\end{hip}
Using results of Theorem \ref{tw1} we will show that Conjecture \ref{conj2} is not satisfied by all odd Schenker primes.

\begin{Convention} We assume that the expression $x \equiv y \pmod{p^k}$ means $v_p(x-y) \geq k$ for prime number $p$ and the integer $k$. The following convention extends relation of equivalence modulo $p^k$ to all rational numbers $x, y$ and integers $k$. Moreover, we set a convention that $\prod_{i=0}^{-1} = 1$.
\end{Convention}

\section{Proof of the main theorem}

Theorem \ref{tw1} recalls a well known fact concerning $p$-adic valuation of a value of polynomial with integer coefficients (see \cite{4}, page 44):

\begin{tw}\label{tw2}
Let $f$ be a polynomial with integer coefficients, $p$ be a prime number and $k$ be a positive integer. Assume that $f(n_0)\equiv 0 \pmod{p^k}$ for some integer $n_0$. Then number of solutions $n$ of the congruence $f(n)\equiv 0 \pmod{p^{k+1}}$, satisfying the condition $n \equiv n_0 \pmod{p^k}$, is equal to:
\begin{itemize}
\item $1$, when $f'(n_0)\not\equiv 0 \pmod{p}$;
\item $0$, when $f'(n_0)\equiv 0 \pmod{p}$ and $f(n_0)\not\equiv 0 \pmod{p^{k+1}}$;
\item $p$, when $f'(n_0)\equiv 0 \pmod{p}$ and $f(n_0)\equiv 0 \pmod{p^{k+1}}$.
\end{itemize}
\end{tw}
Similarity of these theorems is not incidential. Namely, we will show that checking $p$-adic valuation of values of some polynomials is sufficient for computation of the $p$-adic valuation of the Schenker sum. Firstly, note that for any positive integers $d, n$ which are coprime the divisibility of $a_n$ by $d$ is equivalent to the divisibility of $a_{n \mbox{\scriptsize{ mod }}p}$ by $d$:
\begin{equation}\label{eq1}
\begin{split}
a_n &= \sum_{j=0}^{n}\frac{n!}{j!}n^j = \sum_{j=0}^{n}n^{n-j}\prod_{i=0}^{j-1}(n-i) \equiv \\
 &\equiv  \sum_{j=0}^{d-1}n^{n-j}\prod_{i=0}^{j-1}(n-i) =  n^{n-d+2}\sum_{j=0}^{d-1}n^{d-j-2}\prod_{i=0}^{j-1}(n-i) \pmod{d},
\end{split}
\end{equation}
where the equivalnce between the third and the fourth expression follows from the fact that the product of at least $d$ consecutive integers contains an integer $d$ divisible by $d$, hence it is equal to 0 mod $d$. Thus for every $d\in\mathbb{N}_+$ we define the polynomial:
\begin{equation*}
f_d(X):=\sum_{j=0}^{d-1}X^{d-j-2}\prod_{i=0}^{j-1}(X-i).
\end{equation*}
With this notation the formula (\ref{eq1}) can be rewritten in the following way:
\begin{equation}\label{eq2}
a_n \equiv n^{n-d+2}f_d(n) \pmod{d}.
\end{equation}
Let $r=n\pmod{d}$. If $d, n$ are coprime, then the following sequence of equivalences is true:
\begin{equation*}\label{eq6}\begin{split}
& a_n\equiv n^{n-d+2}f_d(n)\equiv 0 \pmod{d} \iff\quad f_d(n)\equiv 0\pmod{d} \\
\iff\quad & f_d(r)\equiv 0\pmod{d} \iff\quad a_r\equiv r^{r-d+2}f_d(r)\equiv 0 \pmod{d}. \\ 
\end{split}\end{equation*}
If $d=p^k$ for some prime number $p$ and positive intiger $k$, then the formula (\ref{eq2}) takes the form:
\begin{equation}\label{eq3}
a_n \equiv n^{n-p^k+2}f_{p^k}(n) \pmod{p^k}.
\end{equation}
We thus see that if $p\nmid n$, then $v_p(a_n)\geq k$ if and only if $v_p(f_{p^k}(n))\geq k$. Moreover, for any $k_1, k_2 \in\mathbb{N}$ satisfying $k_1 \leq k_2$ the following congruence holds:
\begin{equation*}
n^{n-p^{k_1}+2}f_{p^{k_1}}(n) \equiv n^{n-p^{k_2}+2}f_{p^{k_2}}(n) \pmod{p^{k_1}}.
\end{equation*}
Hence, if $p\nmid n$ and $k_1 \leq k_2$, then: 
\begin{equation*}
p^{k_1} \mid f_{p^{k_2}}(n) \iff p^{k_1} \mid f_{p^{k_1}}(n).
\end{equation*}
 If we assume now that $k>1$, then by Fermat's little theorem (in the form \linebreak$n^{p^k} \equiv n \pmod{p}$) and the fact that product of at least $p$ consecutive integers is divisible by $p$ we obtain:

\begin{equation*}\begin{split}
f'_{p^k}(n) &= \sum_{j=0}^{p^k-1}\left[(p^k-j-2)n^{p^k-j-3}\prod_{i=0}^{j-1}(n-i)+n^{p^k-j-2}\sum_{h=0}^{j-1}\prod_{i=0, i\neq h}^{j-1}(n-i)\right] \equiv \\
& \equiv \sum_{j=0}^{2p-1}\left[(-j-2)n^{-j-2}\prod_{i=0}^{j-1}(n-i)+n^{-j-1}\sum_{h=0}^{j-1}\prod_{i=0, i\neq h}^{j-1}(n-i)\right] \pmod{p}.
\end{split}\end{equation*}
The formula above implies the congruence:
\begin{equation}\label{eq4}
f'_{p^{k_1}}(n) \equiv f'_{p^{k_2}}(n) \pmod{p},
\end{equation}
for $k_1, k_2 > 1$ and $p\nmid n$. Let us recall that if $f\in\mathbb{Z}[X]$, then for any $x_0\in\mathbb{Z}$ there exists such a $g\in\mathbb{Z}[X]$ that

\begin{equation*}
f(X) = f(x_0) + (X-x_0)f'(x_0) + (X-x_0)^2g(X).
\end{equation*}
Using the formula (\ref{eq3}) and the equality above for $f=f_{p^2}$, $x_0=n$ and $X=n+p$, we have:

\begin{equation}\label{eq5}
\frac{a_{n+p}}{{(n+p)}^{n+p-p^2+2}}-\frac{a_n}{n^{n-p^2+2}} \equiv f_{p^2}(n+p)-f_{p^2}(n) \equiv pf'_{p^2}(n) \pmod{p^2}.
\end{equation}
If $p\nmid x$, then by Euler's theorem $x^{p^2-p} \equiv 1 \pmod{p^2}$ we can simplify the congruence (\ref{eq5}) and get:

\begin{equation*}
\frac{a_{n+p}}{{(n+p)}^{n+2}}-\frac{a_n}{n^{n-p+2}} \equiv pf'_{p^2}(n) \pmod{p^2}
\end{equation*}
and this leads to the congruence

\begin{equation*}
\frac{1}{p}\left(\frac{a_{n+p}}{{(n+p)}^{n+2}}-\frac{a_n}{n^{n-p+2}}\right) \equiv f'_{p^2}(n) \pmod{p}.
\end{equation*}
Our consideration shows that the following conditions are equivalent:

\begin{equation}\label{eq6}\begin{split}
 &f'_{p^2}(n)\not\equiv 0 \pmod{p} \\
\iff\quad & \frac{1}{p}\left(\frac{a_{n+p}}{{(n+p)}^{n+2}}-\frac{a_n}{n^{n-p+2}}\right)\not\equiv 0\pmod{p} \\ 
\iff\quad &  \frac{a_{n+p}}{{(n+p)}^{n+2}}-\frac{a_n}{n^{n-p+2}}\not\equiv 0\pmod{p^2}  \\ 
\iff\quad & a_{n+p}-\frac{a_n(n+p)^{n+2}}{n^{n-p+2}} \not\equiv 0 \pmod{p^2} \\ 
\iff\quad & a_{n+p}-a_n(n+p)^{n+2}n^{p-n-2} \not\equiv 0 \pmod{p^2}.
\end{split}\end{equation}

Assume now that $p^k\mid a_{n_k}$ for some $n_k\in\mathbb{N}$ not divisible by $p$ and
\begin{equation*}
a_{n_k+p}-a_{n_k}(n_k+p)^{n_k+2}n_k^{p-n_k-2} \not\equiv 0 \pmod{p^2}.
\end{equation*}
Then $p\nmid f'_{p^2}(n_k)$ and by (\ref{eq5}) $p\nmid f'_{p^{k+1}}(n_k)$. Using now the Theorem \ref{tw2} for $f=f_{p^{k+1}}$ we conclude that there exists the unique $n_{k+1}\in\mathbb{Z}$ modulo $p^{k+1}$ satisfying conditions $p^{k+1}\mid a_{n_{k+1}}$ and $n_{k+1} \equiv n_k \pmod{p^k}$. 

By simple induction on $k$ we obtain that if $p\nmid n_1$ and $p \mid a_{n_1}$ together with the condition 
\begin{equation*}
a_{n_1+p}-a_{n_1}(n_1+p)^{n_1+2}n_1^{p-n_1-2} \not\equiv 0 \pmod{p^2},
\end{equation*}
then there exists the unique $n_k$ modulo $p^k$ such that $p^k \mid a_{n_k}$, $n_k \equiv n_1 \pmod{p}$ and
\begin{equation*}
\frac{1}{p}\left(\frac{a_{n_k+p}}{{(n_k+p)}^{n_k+2}}-\frac{a_{n_k}}{n_k^{n_k-p+2}}\right) \equiv
\frac{1}{p}\left(\frac{a_{n_1+p}}{{(n_1+p)}^{n_1+2}}-\frac{a_{n_1}}{n_1^{n_1-p+2}}\right) \pmod{p}.
\end{equation*}
Certainly this statement is true for $k=1$. Now, if we assume that there exists the unique $n_k$ modulo $p^k$ satisfying the conditions in the statement, then there exists the unique $n_{k+1}$ modulo $p^{k+1}$ such that $p^{k+1} \mid a_{n_{k+1}}$, $n_{k+1} \equiv n_k \pmod{p^k}$. Using (\ref{eq6}) we conclude that
\begin{equation*}
\begin{split}
& \frac{1}{p}\left(\frac{a_{n_{k+1}+p}}{{(n_{k+1}+p)}^{n_{k+1}+2}}-\frac{a_{n_{k+1}}}{n_{k+1}^{n_{k+1}-p+2}}\right) 
 \equiv f'_{p^{k+1}}(n_{k+1}) \equiv f'_{p^{k+1}}(n_k) \equiv \\
& \equiv \frac{1}{p}\left(\frac{a_{n_k+p}}{{(n_k+p)}^{n_k+2}}-\frac{a_{n_k}}{n_k^{n_k-p+2}}\right) \equiv  \frac{1}{p}\left(\frac{a_{n_1+p}}{{(n_1+p)}^{n_1+2}}-\frac{a_{n_1}}{n_1^{n_1-p+2}}\right) \pmod{p}.
\end{split}
\end{equation*}
Summing up our discussion we see that the first case in the statement of \linebreak Theorem \ref{tw1} is proved. We prove the rest of the statement now.

Let $p\nmid n_k$ and $p^k\mid a_{n_k}$ and 
\begin{equation*}a_{n_k+p}-a_{n_k}(n_k+p)^{n_k+2}n_k^{p-n_k-2} \equiv 0 \pmod{p^2}.\end{equation*}
Since the last of the conditions above is equivalent to divisibility of $f'_{p^k}(n_k)$ by $p$, the Theorem \ref{tw2} allows us to conclude that:
\begin{itemize}
\item if $p^{k+1} \mid a_{n_k}$, then $p^{k+1} \mid a_n$ for any $n \equiv n_k \pmod{p^k}$;
\item if $p^{k+1} \nmid a_{n_k}$, then $p^{k+1} \nmid a_n$ for any $n \equiv n_k \pmod{p^k}$.
\end{itemize}

We have obtained an useful criterion for behaviour of $p$-adic valuation for numbers $a_n$. In particular, the condition:
\begin{equation*}
a_{n_1+p}-a_{n_1}(n_1+p)^{n_1+2}n_1^{p-n_1-2}\not\equiv 0 \pmod{p^2}
\end{equation*}
is not only sufficient, but also necessary condition for existence of the unique solution modulo $p^k$ of inequality $v_p(a_n) \geq k$ such that $n \equiv n_1 \pmod{p}$.

\section{Solution of conjectures}
First of all let us note that Theorem \ref{tw1} provides the formula $v_2(a_n) = 1$ for every odd positive integer $n$. Indeed:
\begin{equation*}
q_{1,2} = a_3 - a_1\cdot 3^{1+2}\cdot 1^{2-1-2} = 78 - 2\cdot 27 = 24 \equiv 0 \pmod{4}
\end{equation*}
and $a_1 = 2$ and thus $2\nmid n$, then $v_2(a_n) = 1$. This gives an alternative proof of Amdeberhan's, Callan's and Moll's result.

Theorem \ref{tw1} allows us to prove that the Conjecture \ref{conj1} is true by verifying the condition $a_7-a_2\cdot 7^{2+2}\cdot 2^{5-2-2} \not\equiv 0 \pmod{5^2}$. It is easy to check that
\begin{equation*}
a_7-a_2\cdot 7^{2+2}\cdot 2^{5-2-2} = 3309110-10\cdot 2401\cdot 2 = 3261090 \equiv 15 \not\equiv 0 \pmod{25}
\end{equation*}
and the proof of the Conjecture \ref{conj1} is finished.

Let us take the next Schenker prime $p=13$. If $13\nmid n$, then $13\mid a_n$ if and only if $n \equiv 3 \pmod{13}$. Using Theorem \ref{tw1} for $p=13$ and $n_1=3$:
\begin{equation*}\begin{split}a_{16}-a_3\cdot 16^{3+2}\cdot 3^{13-3-2} &= 105224992014096760832-78\cdot 1048576\cdot 6561 =\\
&= 117-78\cdot 100\cdot 139 = -1084083 \equiv 52 \pmod{169},
\end{split}\end{equation*}
we conclude that for every positive natural $k$ there exists the unique solution modulo $13^k$ of inequality $v_{13}(a_n)\geq k$ which is not divisible by $p$ and we know that it is congruent to 3 modulo 13. This implies that if $p=13$, then the Conjecture \ref{conj1} is true.

The Conjecture \ref{conj2} states that for every odd Schenker prime $p$ there exists the unique $n_1 \in\mathbb{N}_+$ less than $p$ such that $p\mid a_{n_1}$ and for this $n_1$ we have:
\begin{equation*}a_{n_1+p}-a_{n_1}(n_1+p)^{n_1+2}n_1^{p-n_1-2}\not\equiv 0 \pmod{p^2}.\end{equation*}
However, it is easy to see that the Conjecture \ref{conj2} is not true in general. Indeed, let us put $p=37$. If $37\nmid n$, then $37\mid a_n$ if and only if $n \equiv 25 \pmod{37}$. However, $37^2\mid a_{62}-a_{25}\cdot 62^{27}\cdot 25^{10}$. Moreover, $a_{25}\equiv 851 = 23\cdot 37 \pmod{37^2}$, thus $v_{37}(a_n)=1$ for any $n \equiv 25 \pmod{37}$. Hence 37-adic valuation of Schenker sums is bounded by one on the set of positive natural numbers not divisible by 37. We can describe it by a simple formula:
\begin{equation*}
v_{37}(a_n) = \begin{cases}
\frac{n-s_{37}(n)}{36}, & \mbox{ when } n\equiv 0 \pmod{37}
\\1, & \mbox{ when } n \equiv 25 \pmod{37}
\\0, & \mbox{ when } n \not\equiv 0, 25 \pmod{37}
\end{cases}.\end{equation*}
Our result shows that the Conjecture \ref{conj2} is false for $p=37$.
Moreover, there exist prime numbers $p$ for which number of solutions modulo $p$ of congruence $a_n \equiv 0 \pmod{p}$, where $p\nmid n$, is greater than 1. Denote this number by $\lambda(p)$ (note that a prime number $p$ is a Schenker prime if and only if $\lambda(p)>0$). According to computations in Mathematica [7], we know that there are 126 Schenker primes among 200 first prime numbers. In the table below we present the solutions of the equation $\lambda(p)=k$ for $k\leq 5$.


\begin{center}
\begin{tabular}{|c|p{10.5 cm}|} \hline
$\lambda(p)$ & $p$ \\ \hline
1 & 5, 13, 23, 31, 37, 43, 47, 53, 59, 61, 71, 79, 101, 103, 107, 109, 127, 137, 157, 163, 173, 229, 241, 251, 257, 263, 317, 337, 349, 353, 359, 397, 421, 431, 487, 499, 503, 521, 547, 571, 577, 587, 617, 619, 641, 653, 661, 727, 733, 757, 797, 811, 821, 829, 881, 883, 937, 947, 967, 977, 991, 1013, 1031, 1039, 1091, 1097, 1123, 1163, 1181, 1213 \\
2 & 41, 149, 181, 191, 199, 211, 271, 283, 293, 311, 367, 383, 401, 409, 419, 439, 461, 523, 541, 563, 569, 607, 613, 647, 673, 691, 709, 761, 787, 827, 929, 941, 983, 997, 1021, 1051, 1061, 1087, 1117, 1151, 1153, 1223 \\
3 & 179, 197, 223, 277, 509, 601, 683, 743, 887, 1201 \\
4 & --- \\
5 & 593 \\ \hline
\end{tabular}
\begin{center}Table 1\end{center}
\end{center}


\section{Questions}
Although Conjecture \ref{conj2} is not true, we do not know if 2 and 37 are the only primes $p$ such that $p\mid a_n$ and $q_{n,p}=a_{n+p}-a_n(n+p)^{n+2}n^{p-n-2}\equiv 0\pmod{p^2}$ for some $n\in\mathbb{N}_+$ which is not divisible by $p$. According to a numerical computations, they are unique among all primes less than 30000. The results above suggest to formulate the following questions:
\begin{pyt}
Is there any Schenker prime greater than $37$ for which there exists $n\in\mathbb{N}_+$ such that $p\nmid n$, $p\mid a_n$ and $q_{n,p}\equiv 0\pmod{p^2}$?
\end{pyt}
\begin{pyt}
Are there infinitely many Schenker primes $p$ for which there exisits $n\in\mathbb{N}_+$ such that $p\nmid n$, $p\mid a_n$ and $q_{n,p}\equiv 0\pmod{p^2}$?
\end{pyt}

In the light of the results presented in the table some natural questions arise:
\begin{pyt}
Are there infinitely many Schenker primes $p$ for which $\lambda(p)>1$?
\end{pyt}
\begin{pyt}
Let $m$ be a positive integer. Is there any Schenker prime $p$ such that $\lambda(p)\geq m$?
\end{pyt}

In section 1 of this paper we presented a short proof of the infinitude of set of Schenker primes. In view of this fact it is natural to ask:
\begin{pyt}
Are there infinitely many primes which are not Schenker primes?
\end{pyt}

\section*{Acknowledgements}

I wish to thank my MSc thesis advisor Maciej Ulas for many valuable remarks concerning presentation of results. I would like also to thank Maciej Gawron for help with computations and Tomasz Pełka for help with edition of the paper.

\end{document}